\documentclass{article}
\usepackage{amsmath,amsthm,amssymb,amscd}

\newcommand{\ga}{\alpha}
\newcommand{\gb}{\beta}

\newcommand{\gw}{\omega}

\newcommand{\gS}{\Sigma}
\newcommand{\gs}{\sigma}

\newcommand{\forkindep}[1][]{%
  \mathrel{
    \mathop{
      \vcenter{
        \hbox{\oalign{\noalign{\kern-.3ex}\hfil$\vert$\hfil\cr
              \noalign{\kern-.7ex}
              $\smile$\cr\noalign{\kern-.3ex}}}
      }
    }\displaylimits_{#1}
  }
}

\newcommand{\liff}{\leftrightarrow}

\newcommand{\cantor}{2^\gw}

\newcommand{\bintree}{2^{<\gw}}
\newcommand{\gwtree}{\gw^{<\gw}}

\newcommand{\dom}{\mathrm{dom}}

\newcommand{\power}{\mathcal{P}}

\newcommand{\hvd}{\mathrm{HOD}}
\newcommand{\vd}{\mathrm{OD}}

\newcommand{\coll}{\mathrm{Coll}}

\newtheorem{theorem}{Theorem}[section]

\newtheorem{claim}[theorem]{Claim}

\newtheorem{fact}[theorem]{Fact}
\newtheorem{proposition}[theorem]{Proposition}

\theoremstyle{definition}
\newtheorem{definition}[theorem]{Definition}
\newtheorem{example}[theorem]{Example}
\newtheorem{question}[theorem]{Question}

\newtheorem{conjecture}[theorem]{Conjecture}

\title{More on the Boolean Prime Ideal Theorem\footnote{2020 AMS subject classification 03E25, 22F05. Keywords: Boolean Prime Ideal Theorem, Axiom of Dependent Choices, Solovay model}}

\author{
	Jacob Kowalczyk\\
	University of Florida\\
	kowalczykj@ufl.edu
	\and
	Jind{\v r}ich Zapletal\\
	University of Florida\\
	zapletal@ufl.edu}

\begin{document}
	\maketitle
	
	\begin{abstract}
		We prove the consistency of  Zermelo--Fraenkel set theory with the Axiom of Dependent Choices, no Vitali sets and a large fragment of the Boolean Prime Ideal Theorem.
	\end{abstract}
	
	\section{Introduction}

The Boolean Prime Ideal Theorem (BPI) is one of the most prominent consequences of the Axiom of Choice (AC). It asserts that every Boolean algebra carries an ultrafilter, or alternately, that every consistent propositional theory has a consistent completion. It is well-known that in the choiceless set theory ZF, BPI does not imply AC \cite{halpern:lauchli}. However, in the presence of the Axiom of Dependent Choices (DC), the picture changes, and one of the oldest open questions in choiceless set theory \cite{pincus:original} asks whether in the base theory ZF+DC, BPI is equivalent to AC. In view of the accumulated experience with these statements, one can form a natural conjecture.

\begin{conjecture}
\label{mainconjecture}
It is consistent with ZF+DC that the Boolean Prime Ideal Theorem holds and there is no Vitali set in the real line.
\end{conjecture}

\noindent Here, a Vitali set is a subset of the real line which for every real contains exactly one of its rational translates. In this paper, we show that in ZF+DC, a large fragment of BPI is consistent with nonexistence of Vitali sets. 

\begin{theorem}
\label{maintheorem}
Suppose that there is an inaccessible cardinal. Then there is an inner model of a generic extension satisfying the following:

\begin{enumerate}
\item ZF+DC;
\item there are no Vitali sets;
 \item every consistent $K_\gs$ propositional theory on Polish, locally compact set of atomic formulas has a consistent completion.
\end{enumerate}
\end{theorem}

\noindent The fragment of BPI quoted in Theorem~\ref{maintheorem} implies for example that every $F_\gs$-ideal on $\gw$ can be extended to a maximal ideal (Example~\ref{ultrafilterexample}), or that de Bruijn--Erd{\H o}s theorem holds for $K_\gs$ graphs on locally compact Polish spaces (Example~\ref{coloringexample}). The fragment can easily be extended to include a number of other statements, such as ``the set of analytic subsets of $\cantor$ can be linearly ordered''.

Our framework is a streamlined presentation of compactly balanced forcing \cite[Section 9.2]{z:geometric}. In order to state the proofs in the most efficient way, the whole paper uses the geometric axiomatization of the Solovay model \cite{z:reloadedA}. Our models are then all balanced forcing extensions of the Solovay model, as axiomatized in \cite{z:reloadedB}. 

In Section~\ref{0section}, we identify the class of balanced theories, which are propositional theories to which it is possible to add a consistent complete extension by a rather canonical forcing (a Boolean balanced forcing of Definition~\ref{maindefinition}) in a controlled way. Section~\ref{balancedsection} contains a long list of natural balanced theories, Section~\ref{imbalancedsection} contains a list of natural theories which are not balanced. In Section~\ref{preservationsection} we use a suitable infinitary partition theorem to show (as a special case of the powerful Theorem~\ref{preservationtheorem}) that Boolean balanced forcings do not add Vitali sets. This makes it possible to prove Theorem~\ref{maintheorem} at the end of the section. The appendix Section~\ref{appendixsection} spells out the geometric axiomatization of the Solovay model for reference purposes.

While the methods of the present paper provide an insight into Conjecture~\ref{mainconjecture}, they can hardly be used to provide a complete affirmative solution. The simplest questions remaining open are the following.

\begin{question}
Is ZF+DC+``$\power(\mathbb{R})$ can be linearly ordered" consistent with nonexistence of Vitali sets? See also \cite{schilhan:ordering, pincus:addingdc}.
\end{question}

\begin{question}
Is ZF+DC+``there is a maximal ideal extending the asymptotic density zero ideal on $\gw$'' consistent with nonexistence  of Vitali sets?
\end{question}

\noindent The most obvious technical open problem concerning the methods of the present paper is the following.

\begin{question}
(In the Solovay model) Characterize those ideals $I$ on $\gw$ for which the ultrafilter theory of $I$ is balanced.
\end{question}

\noindent For the notation and terminology, we use the set theoretic standard of \cite{jech:newset}, the geometric axiomatization of the Solovay model of \cite{z:reloadedA}, and the axiomatization of balanced forcing of \cite{z:reloadedB}. A Vitali set is a set of reals which to every real contains exactly one of its rational shifts. DC denotes the Axiom of Dependent Choices, which is the statement that every partially ordered set contains either a minimal element or an infinite strictly descending sequence.

	\section{Propositional theories}
	\label{0section}

A search for completions of consistent propositional theories comes up naturally in many different contexts. We include a brief list of examples to refer to later in the paper.

\begin{example}
\label{ultrafilterexample}
	(Ultrafilter theory of an ideal)
Let $A$ be a set and $I$ an ideal on it. Consider the theory $T$ using subsets of $A$ as variables in its propositional language and including the sentences $A$, for all sets $b, c\subseteq A$ the sentence $(b\land c)\to (b\cap c)$, for sets $b\subseteq c\subseteq A$ the sentence $b\to c$, for sets $b\subseteq A$ the sentence $b\lor (A\setminus b)$ and for sets $b\in I$ the sentence $\lnot b$. The completions correspond to ultrafilters on $A$ disjoint from $I$. Note that in this theory, every formula is provably equivalent to an atomic one.
\end{example}

\begin{example}
\label{coloringexample}
	(Coloring theory of a graph)
	Let $n\in\gw$ be a natural number and $\langle V, G\rangle$ be a graph all of whose finite subgraphs have chromatic number $\leq n$. Consider the theory $T$ in language $\mathcal{L}_T$, whose propositional variables come from the set $V\times n$, containing the sentences $\bigvee_{i\in n}\langle v, i\rangle$  and $\lnot(\langle v, i\rangle\land\langle v, j\rangle)$ for each $v\in V$ and distinct numbers $i, j\in n$, and $\lnot (\langle v_0, i\rangle\land\langle v_1, i\rangle)$ for each $\{v_0, v_1\}\in G$ and $i\in n$. Consistent completions correspond to $G$-colorings.
\end{example}

\begin{example}
	\label{linearizationexample}
	(Linearization theory) Let $\langle A, \leq\rangle$ be a partial order. Let $T$ be a propositional theory on the set $A^2$ of propositional variables containing the following statements:
	
	\begin{enumerate}
		\item $\langle a, b\rangle$ for any $\langle a, b\rangle\in\leq$;
		\item exactly one of $\langle a, b\rangle$ and $\langle b, a\rangle$ holds, for any distinct $a, b\in A$;
		\item $\langle a, b\rangle$ and $\langle b, c\rangle$ implies $\langle a, c\rangle$ for any $a, b, c\in A$.
	\end{enumerate}
	
	\noindent Consistent completions correspond to linearizations of $\leq$.
\end{example}

\noindent There are many ways in which a completion of a given propositional theory could be added by forcing. Among them, the current paper investigates the most canonical option only.
	
	\begin{definition}
		Let $T$ be a consistent propositional theory in language (set of propositional variables) $\mathcal{L}_T$. The \emph{associated poset} $P_T$ consists of well-orderable sets $p$ of formulas in the language $\mathcal{L}_T$ such that $T\cup p$ is consistent. The ordering is that of reverse inclusion.
	\end{definition}

\noindent The associated poset is designed to add (as the union of the generic filter) a consistent completion of $T$. One important observation is that under DC (which is part of our base theory), countable unions of well-orderable sets are well-orderable, so the poset $P_T$ is $\gs$-closed. In addition, $\gs$-closed forcings preserve DC, so the poset $P_T$ is a suitable tool for adding an instance of BPI while preserving DC. In general though, it is difficult to assert more control over this poset. The following definition, reminiscent of funicular preorders of \cite{azul:funicular}, isolates a tool which is useful in this direction

\begin{definition}
	\label{maindefinition}
	Let $T$ be a consistent propositional theory. The theory $T$ is \emph{balanced} if for all sets $x, y_0, y_1$ of ordinals such that $T\in\vd_x$ and $y_0\forkindep_xy_1$ and all formulas $\phi_0\in\vd_{xy_0}, \phi_1\in\vd_{xy_1}$ in the language $\mathcal{L}_T$, if $T\vdash\phi_0\to\phi_1$ then there is an interpolant $\psi\in\vd_x$, a formula such that $T\vdash\phi_0\to\psi$ and $T\vdash\psi\to\phi_1$.
	A poset is \emph{Boolean balanced} if it is a poset associated with a balanced theory.
\end{definition}

\noindent The main theorem of this section shows that the adjective in Definition~\ref{maindefinition} is appropriate.

\begin{theorem}
	\label{balancetheorem}
	If $T$ is a balanced theory, then the associated poset $P_T$ is balanced. In addition, if $x$ is a set of ordinals such that $T\in\vd_x$,
	
	\begin{enumerate} 
		\item every condition in $P_T\cap\vd_x$ can be extended to a complete theory in the language $\mathcal{L}_T\cap\vd_x$ consistent with $T$;
		\item if $q$ is a complete theory as in the previous item then $O_q=\{r\in P_T\colon r\leq q\}\subset P_T$ is an open set balanced over $x$;
		\item every set balanced over $x$ intersects exactly one open set from the previous item.
	\end{enumerate}
\end{theorem}

\noindent As for the notation, observe that a complete theory in the language $\mathcal{L}_T\cap\vd_x$ consistent with $T$ is an element of $\vd_x$ is in fact a condition in $P$, and it is $\leq$-minimal condition in $P_T\cap\vd_x$.

\begin{proof}
	Let $x$ be a set of ordinals such that $T\in\vd_x$. For (1), note that every condition in $P\cap\vd_x$ is in fact a subset of $\vd_x$ by Fact~\ref{facticfact}(1). Thus, using the canonical well-ordering of $\vd_x$, any condition in $P\cap\vd_x$ can be extended to a complete theory in the language $\mathcal{L}_T\cap\vd_x$ which is consistent with $T$ and an element of $\vd_x$. 
	
	For (2), let $q\in\vd_x$ be such a complete theory in $\vd_x$ in the language $\mathcal{L}\cap\vd_x$ consistent with $T$. Let $y_0, y_1$ be sets of ordinals such that $y_0\forkindep[x] y_1$, and $r_0\in\vd_{xy_0}$ and $r_1\in\vd_{xy_1}$ be conditions in the open set $O_q\subset P_T$; without loss, assume that both are closed under conjunctions. Suppose towards a contradiction that $r_0$ and $r_1$ have no common lower bound. This means that there are formulas $\phi_0\in r_0$ and $\phi_1\in r_1$ such that $r_0\land r_1$ is inconsistent with $T$; in other words, $T\vdash\phi_0\to\lnot\phi_1$.  By the balance assumption on the theory $T$, there is an interpolant $\psi\in\vd_x$ such that $T$ proves $\phi_0\to\psi$ and $\psi\to\lnot\phi_1$. Now, the completeness assumption on $q$ implies that either $\psi\in q$ or $\lnot\psi\in q$. The former option would mean that $r_1$ is inconsistent, while the latter option would mean that $r_0$ is inconsistent. In both cases, a contradiction is reached.
	
	For (3), suppose that $O\subset P$ is a set balanced over $x$; we must find $q\in P$ as in (2) such that $O\cap O_q\neq 0$. Note that the set of all $\vd_x$ formulas in the language $\mathcal{L}_T$ is well-orderable, so every condition can be extended to a condition $r$ such that for every $\vd_x$-formula $\phi$, either $\phi\in r$ or $\lnot\phi\in r$. By the extension property of $\forkindep$, there are sets $y_0, y_1$ of ordinals such that $y_0\forkindep[x]y_1$ and conditions $r_0\in\vd_{xy_0}$ and $r_1\in\vd_{xy_1}$ in $O$ such that for every $\vd_x$-formula $\phi$, either $\phi\in r_0$ or $\lnot\phi\in r_0$, and  either $\phi\in r_1$ or $\lnot\phi\in r_1$. The conditions $r_0, r_1$ are compatible in the poset $P$, so the decision regarding $\phi$ must be the same in $r_0$ and $r_1$. So, let $q=r_0\cap\vd_x=r_1\cap\vd_x$. This is a complete theory in the language $\mathcal{L}_T\cap\vd_x$; in addition, Fact~\ref{facticfact}(2) shows that $q\in\vd_{xy_0}\cap\vd_{xy_1}=\vd_x$. Since both $r_0, r_1\in O$ are conditions stronger than $q$, it follows that $O\cap O_q\neq 0$ as desired.
	
	Items (1-3) together imply the balance of the poset $P$.  
\end{proof}

\noindent While in many natural cases, the status of balance of a given propositional theory is easy to check, in many other cases it is quite challenging. The following two sections are devoted to a study of many particular cases of the central question:

\begin{question}
	When is a consistent propositional theory $T$ balanced? Can a consistent theory be always extended to a consistent balanced theory, possibly in a larger language?
\end{question}

\section{Balanced theories}
\label{balancedsection}

In this section, we provide positive results: criteria which imply that certain classes of propositional theories are balanced. For the ultrafilter theories, we supply only a rather restrictive criterion, and negative results in Section~\ref{imbalancedsection} suggest that nothing much better may be available.

\begin{theorem}
\label{welltheorem}
Let $A$ be a well-orderable set, and let $I$ be an ideal on $A$. Suppose that

\begin{enumerate}
\item[(*)] $I$ is generated by a well-orderable collection of sets.
\end{enumerate}

\noindent Then the ultrafilter theory $T$ of $I$ is balanced.
\end{theorem}

\noindent We do not know if the criterion here is in fact necessary and sufficient for the conclusion.

\begin{proof}
	Let $x$ be a set of ordinals such that $A, I\in\vd_x$. Fact~\ref{facticfact}(1) shows that $A\subset\vd_x$ holds. We also have the following:
	
	\begin{claim}
		\label{helpyclaim}
		The ideal $I$ is generated by sets in $I\cap\vd_x$.
	\end{claim}
	
	\begin{proof}
Suppose towards a contradiction that this fails. Then there is a set $y_0$ of ordinals and a set $a_0\in I\cap\vd_{xy_0}$ which is not a subset of any set in $I\cap \vd_x$. By the extension property, there is a set $y_1$ of ordinals such that $y_0\forkindep[x]y_1$ and $\vd_{xy_1}$ contains a well-orderable set $J\subseteq I$ which generates $I$. By Fact~\ref{facticfact}(1), $J\subset\vd_{xy_1}$ holds. Thus, there is a set $a_1\in J$ such that $a_0\subseteq a_1$ holds. By the well-orderability assumption on the set $A$, it is the case that $A\subset\vd_x$, so $a_0$ and $A\setminus a_1$ are disjoint subsets of $\vd_x$ in independent extensions. Apply Fact~\ref{facticfact}(3) to find a set $b\subset A$ in $\vd_x$ such that $a_0\subseteq b\subseteq a_1$. Then $b\in I$ must hold, contradicting the choice of $a_0$.
	\end{proof}
	
\noindent Now, suppose that $x, y_0, y_1$ are sets of ordinals such that $A, I\in\vd_x$ and $y_0\forkindep[x]y_1$, and $\phi_0\in\vd_{xy_0}$ and $\phi_1\in\vd_{xy_1}$ are formulas in the ultrafilter language such that $T\vdash\phi_0\to\phi_1$. There are sets $a_0, a_1\subset A$ obtained by replacing Boolean operations with set theoretic ones such that $T\vdash\phi_0\liff a_0$ and $T\vdash\phi_1\liff a_1$. It must be the case that $a_0\setminus a_1\in I$. By Claim~\ref{helpyclaim}(1) there must be a set $c\in I\cap\vd_x$ such that $a_0\setminus c\subset a_1$. By Fact~\ref{facticfact}(3) there is a set $d\in\vd_x$ such that $a_0\setminus c\subseteq d\subseteq a_1$. Then the propositional variable $d$ is the desired interpolant.
\end{proof}

\begin{example}
	The ultrafilter theory of the ideal of finite sets on $\gw$ is balanced.
\end{example}

\begin{example}
	The ultrafilter theory of the ideal of countable sets on $\gw_1$ is balanced.
\end{example}

\begin{example}
	In the Solovay model, every closed unbounded subset of $\gw_1$ has a club subset in $\hvd$. (We do not know if this follows from the axiomatization.) This means that the ideal $I$ of nonstationary subsets of $\gw_1$ has a well-orderable basis, so the ultrafilter theory of $I$ on $\gw_1$ is balanced.
\end{example}

\noindent The most powerful technical tool in this paper is a criterion for balance of $K_\gs$-theories on locally compact Polish spaces. Here, we must first explain topologization of theories used in this paper.

\begin{definition}
	Let $X$ be a Polish space, viewed as a space of atomic propositional variables. Let $S$ be a finite set of logical connectives and parentheses of that language. $X^*$ is the set of all formulas in this language. To topologize it, consider the space $X\cup S$ where $S$ is equipped with discrete topology, the spaces $(X\cup S)^d$ for different $d\in\gw$ equipped with the product topology, and topologize $(X\cup S)^{<\gw}$ as a disjoint sum $\bigcup_d(X\cup S)^d$. $X^*$ is then considered as a closed subset of $(X\cup S)^{<\gw}$.
\end{definition}

\noindent The main focus in this paper is the class of $K_\gs$-theories on locally compact Polish spaces. The main descriptive reason for this restriction is the following simple observation.

\begin{proposition}
If $X$ is a locally compact space, then so is $X^*$. In such a space, deductive closures of $K_\gs$-theories are $K_\gs$ again. 
\end{proposition}

\begin{proof}
The local compactness of $X^*$ is obvious. For the complexity of the deductive closure, let $T\subseteq X^*$ be a $K_\gs$-theory. Express $X=\bigcup_nK_n$ and $T=\bigcup_nC_n$ as increasing unions of compact sets. For each number $n\in\gw$ the set $D_n$ of all proofs of length $n$ from $T$, which use only formulas of length $\leq n$, only atomic variables in the set $K_n$ and as for formulas in $T$ only use formulas in $C_n$ is compact as a subset of $((X^*)\cap (X\cup S)^{\leq n})^{\leq n}$ as can be immediately seen by the converging subsequence criterion. The deductive closure of $T$ is then the union of the projections of the sets in $D_n$ into their last coordinates. Projections of compact sets are compact, a key use of the complexity assumptions.
\end{proof}

\begin{theorem}
\label{ksigmacriterion}
Let $T$ be a deductively closed $K_\gs$-theory on a locally compact space $X$. Suppose that

\begin{enumerate}
\item[(*)] there are compact sets $U_n$ for $n\in\gw$ such that  $T=\bigcup_nU_n$ and for every $n\in\gw$ and all compact sets $K_0, K_1\subset X^*$, if for all $\phi_0\in K_0$ and $\phi_1\in K_1$ the implication $\phi_0\to\phi_1$ belongs to $U_n$, then there is a formula $\psi$ such that for all $\phi_0\in K_0$ and $\phi_1\in K_1$ the implications $\phi_0\to\psi$ and $\psi\to\phi_1$ belong to $T$.
\end{enumerate}

\noindent Then $T$ is balanced.
\end{theorem}
	
\begin{proof}
Suppose that $z, y_0, y_1$ are sets of ordinals such that $X, T, \langle U_n\colon n\in\gw\rangle\in\vd_z$ and $y_0\forkindep_zy_1$, and suppose that $\phi_0\in\vd_{zy_0}$ and $\phi_1\in\vd_{zy_1}$ are formulas in $X^*$ such that the implication $\phi_0\to\phi_1$ belongs to $T$. There must be a number $n$ such that the implication belongs to $U_n$.

\begin{claim}
There is a closed set $K_0\subseteq X^*$ in $\vd_z$ such that $\phi_0\in K_0$ and for every $\psi_0\in K_0$, the implication $\psi_0\to\phi_1$ belongs to $U_n$.
\end{claim}

\begin{proof}
Let $W_0$ be the set of all basic open subsets of $X^*$ which contain $\phi_0$, and let $W_1$ be the set of all basic open subsets of $X^*$ which contain no formula $\phi'_0$ such that $\phi'_0\to\phi_1\in U_n$. These are obviously disjoint subsets of $\vd_z$, $W_0\in\vd_{zy_0}$ and $W_1\in\vd_{zy_1}$, so by Fact~\ref{facticfact}(3) there must be a set $B\in\vd_z$ consisting of basic open subsets of $X^*$ such that $A_0\cap B=0$ and $A_1\subseteq B$. Write $K_0=X\setminus \bigcup B\in\vd_z$. As $B$ is disjoint from $A_0$, it is clear that $\phi_0\in K_0$. As $A_1\subseteq B$ and the set $U_n$ is closed, it follows that for every formula $\psi_0\in K_0$, $\psi_0\to\phi_1\in U_n$ as desired.
\end{proof}

\noindent A symmetric argument yields

\begin{claim}
There is a closed set $K_1\subseteq X^*$ in $\vd_z$ such that $\phi_1\in K_1$ and for every $\psi_1\in K_1$, the implication $\phi_0\to\psi_1$ belongs to $U_n$.
\end{claim}

\noindent Use the locally compact assumption to shrink the sets $K_0, K_1$ to compact. Further shrinking $K_0$ to the compact set $\{\psi_0\in K_0\colon\forall \psi_1\in K_1\ \psi_0\to\psi_1\in U_n\}$  and $K_1$ to the compact set $\{\psi_1\in K_1\colon\forall \psi_0\in K_0\ \psi_0\to\psi_1\in U_n\}$, we may achieve the situation where all implications from an element of $K_0$ to an element of $K_1$ belong to the set $U_n$.

Now, use the assumption (*) to conclude that there is a formula $\psi$ such that for all $\psi_0\in K_0$ and all $\psi_1\in K_1$, the implications $\psi_0\to\psi$ and $\psi\to\psi_1$ belong to $T$. This is a ${\mathbf\Sigma}^1_2$-statement; by Shoenfield's absoluteness, such a formula $\psi$ must exist in $\vd_z$. The balance of the theory $T$ follows.
\end{proof}

\noindent It may seem that the criterion in Theorem~\ref{ksigmacriterion} is too technical or uncommon to be useful. However, there is a canonical procedure for extending a given consistent $K_\gs$-theory on a locally compact space to one which satisfies (*). This is described in the following theorem.

\begin{theorem}
\label{completiontheorem}
Let $X$ be a locally compact Polish space and $T\subset X^*$ be a consistent $K_\gs$ theory on $X$. There is a locally compact Polish space $Y\supset X$ and a consistent $K_\gs$ theory $S\supset T$ on $Y$ satisfying (*) of Theorem~\ref{ksigmacriterion}. 
\end{theorem}

\begin{proof}
By recursion on $n$ build locally compact spaces $X_n$ and deductively closed $K_\gs$ consistent theories $T_n\subset X_n^*$ expressed as increasing unions $T_n=\bigcup_mU_{mn}$ of compact sets such that

\begin{itemize}
\item $X_0=X$ and $T_0=T$;
\item $X_n\subseteq X_{n+1}$ is a clopen set, $T_n\subseteq T_{n+1}$;
\item $U_{0n}$ contains $U_{ni}$ for all $i\in n$;
\item whenever $K_0, K_1\subset X_n^*$ are compact sets such that for any formulas $\phi_0\in K_0$ and $\phi_1\in K_1$ the implication $\phi_0\to\phi_1$ belongs to $U_{0n}$, then there is a formula $\psi\in X_{n+1}^*$ such that for every $\phi_0\in K_0$ and every $\phi_1\in K_1$ the implications $\phi_0\to\psi$ and $\psi\to\phi_1$ belong to $T_{n+1}$.
\end{itemize}

\noindent Once this is done, consider the space $X_\gw=\bigcup_nX_n$, the theory $T_\gw=\bigcup_nT_n$, and the sets $U_n=U_{0n}$ for $n\in\gw$. It is clear that $X_\gw$ is a locally compact Polish space, the theory $T_\gw$ is consistent and deductively closed, and $T_\gw$ is an increasing union $\bigcup_nU_n$ of compact sets. In addition, if $n\in\gw$ is a number and $K_0, K_1\subset X_\gw^*$ are compact sets such that for any formulas $\phi_0\in K_0$ and $\phi_1\in K_1$ the implication $\phi_0\to\phi_1$ belongs to $U_{0n}$, then $K_0, K_1\subset X_n^*$ holds, and the last item of the recursion demands shows that there is an interpolant as required in (*).

To complete the proof of the theorem, we must show how to perform the recursion step. Suppose that $X_n$, $T_n$, and $U_{0n}$ have been constructed. Consider the locally compact space $Z_n=K(X^*_n)^2$ and its subset $Y_n=\{\langle K_0, K_1\rangle\colon \forall\phi_0\in K_0\forall \phi_1\in K_1\ \phi_0\to\phi_1\in U_{0n}\}$. An inspection reveals that the set $Y_n\subseteq Z_n$ is closed. Now, let $X_{n+1}$ be a disjoint union of $X_n$ and $Y_n$, and consider the theory $S_n$ using the elements of $Y_n$ as new atomic variables, containing the statements $\phi_0\to\langle K_0, K_1\rangle$ and $\langle K_0, K_1\rangle\to\phi_1$ for all $\phi_0\in K_0$ and $\phi_1\in K_1$ and $\langle K_0, K_1\rangle\in U_{0n}$. The theory $T_n\cup S_n$ is consistent: it is even conservative over $T_n$ in that every truth assignment on $X_n$ making all formulas in $T_n$ true can be extended to a truth assignment making every atomic formula $\langle K_0, K_1\rangle\in Y_n$ true just in case at least one formula in $K_0$ is true, which in turn makes all formulas of $S_n$ true. An inspection reveals that $S_n\subset X^*_{n+1}$ is closed, so $T_n\cup S_n$ is $K_\gs$, and so is its deductive closure $T_{n+1}$. Finally, express $T_{n+1}$ as an increasing union $\bigcup_mU_{mn+1}$ in such a way that $U_{0n+1}$ contains $U_{n+1,i}$ for all $i\in n+1$. This completes the recursion step and the proof.
\end{proof}

\begin{example}
\label{fsigmaexample}
	Let $I$ be an $F_\gs$-ideal on $\gw$. There is a Boolean balanced forcing adding an ultrafilter on $\gw$ disjoint from $I$.
\end{example}

\begin{proof}
	Let $X$ be the space $\power(\gw)$ with the usual compact topology. Each element $x\in X$ is viewed as a propositional variable. Consider the theory $T$ consisting of all statements $x\lor (\gw\setminus x)$, $x_1\to x_0$ if $x_1\subseteq x_0$, $x_0\land x_1\to (x_0\cap x_1)$, and $\lnot x$ whenever $x\in I$. This is clearly a $K_\gs$-theory by the complexity assumption on $I$. Theorem~\ref{completiontheorem} completes the proof.
\end{proof}

\begin{example}
	Let $G$ be a $K_\gs$ graph on a locally compact Polish space $X$ and $n\in\gw$ be a number such that every finite subgraph of $G$ has chromatic number $\leq n$. Then there is a Boolean balanced poset adding a coloring of $G$ with at most $n$ colors.
\end{example}

\begin{proof}
	The coloring theory of the graph is $K_\gs$. Theorem~\ref{completiontheorem} concludes the proof.
\end{proof}

\noindent For linearization theories we also have a practical balance criterion.

\begin{theorem}
\label{linearizationtheorem}
Let $\langle A, \leq\rangle$ be a partial order. Suppose that 

\begin{enumerate}
\item[(*)] there is a function $F\colon A\to\power(A)$ whose range consists of well-orderable sets, such that for every $a_0\leq a_1\in A$ there is $b\in F(a_0)\cap F(a_1)$ with $a_0\leq b\leq a_1$.
\end{enumerate}

\noindent Then the linearization theory of $A$ is balanced.
\end{theorem}

\begin{proof}
Let $T$ denote the linearization theory. Note that a formula $\phi$ in the language of $T$ is consistent with $T$ if there is a finite set $B\subset A$ which contains all elements of $A$ mentioned in $\phi$ and a linearization $\preceq$ of $\leq\restriction B$ which satisfies $\phi$. To see this, observe that every linearization on a finite subset of $A$ can be extended to a linearization on any larger finite subset of $A$.

Towards the proof of balance of $T$, suppose that (*) holds. Suppose that $x, y_0, y_1$ are sets of ordinals such that $F, A\in\vd_x$ and $y_0\forkindep[x]y_1$. 
To begin, observe that for every $a_0\in A\cap\vd_{xy_0}$ and $a_1\in A\cap\vd_{xy_1}$ such that $a_0\leq a_1$ (or $a_0\geq a_1$) there is $b\in\vd_x$ such that $a_0\leq b\leq a_1$ (or $a_0\geq b\geq a_1$). To see this, use the properties of the function $F$ to find such $b$ in the intersection $F(a_0)\cap F(a_1)$. Now, by Fact~\ref{facticfact}(1), $F(a_0)\subseteq \vd_{xy_0}$ and $F(a_1)\subseteq\vd_{xy_1}$ both hold, as the range of $F$ consists of well-orderable sets. In addition, Fact~\ref{facticfact}(2) shows that $F(a_0)\cap F(a_1)\subseteq\vd_x$ must hold, so $b\in\vd_x$ holds as desired.

Now, suppose that $\phi_0\in\vd_{xy_0}$ and $\phi_1\in\vd_{xy_1}$ are formulas in the language of $T$. Let $B_0\in\vd_{xy_0}$ and $B_1\in\vd_{xy_1}$ be finite sets such that $B_0$ contains all elements of $A$ mentioned in $\phi_0$, $B_1$ contains all elements of $A$ mentioned in $\phi_1$, $B_0\cap B_1=B_0\cap\vd_x=B_1\cap\vd_x$, and for any two elements $a_0\in B_0$ and $a_1\in B_1$ such that $a_0\leq a_1$ there is an interpolant $b\in B_0\cap B_1$ such that $a_0\leq b\leq a_1$. Such finite sets can be immediately found by the previous paragraph. We will show that either there is an interpolant $\psi$ in the language $(B_0\cap B_1)^2$ between $\phi_0$ and $\phi_1$, or $\phi_0\land\lnot\phi_1$ is consistent with $T$. This will prove the balance of $T$ as all formulas in the language $(B_0\cap B_1)^2$ belong to $\vd_x$.

Suppose that no interpolant $\psi$ as above exists. Let $S$ be the set of all $T$-consequences of $\phi$ in the language $(B_0\cap B_1)^2$. There are only finitely many formulas in this language up to logical equivalence. Their conjunction does not imply $\lnot\phi$, so there is a linear order $\preceq_1$ on $B_1$ extending $\leq$ which satisfies $S$ and $\lnot\phi_1$. By the definition of $S$, there is a linear order $\preceq_0$ on $B_0$ extending $\leq\cup (\preceq_1\restriction B_0\cap B_1)$ which satisfies $\phi_0$. The key point now is that the relation $\preceq_0\cup\preceq_1\cup\leq\restriction (B_0\cup B_1)$ does not contain any non-trivial cycles. To see this, observe that any such cycle would have to contain elements in both $B_0\setminus B_1$ and $B_1\setminus B_0$, but the number of such elements could be always decreased by using interpolants in $B_0\cap B_1$. Any acyclic, asymmetric relation on a finite set can be extended to a linear order. Choose any linear order $\preceq$ extending  $\preceq_0\cup\preceq_1\cup\leq\restriction (B_0\cup B_1)$ and note that it satisfies $\phi_0\land\lnot\phi_1$, proving that the conjunction is consistent with $T$.
\end{proof}

\begin{example}
Let $A$ be any set and $\leq$ be the trivial partial order on it. The linearization theory of $\leq$ is balanced.
\end{example}

\begin{proof}
To verify the interpolation criterion in Theorem~\ref{linearizationtheorem}, just let $F(a)=\{a\}$ for every element $a\in A$.
\end{proof}

\noindent As a final positive result in this section, we show that Boolean balanced forcings are closed under a type of product operation.

\begin{theorem}
	\label{producttheorem}
	Let $\{T_i\colon i\in I\}$ be a collection of balanced theories in pairwise disjoint languages $\mathcal{L}_i$. Then $T=\bigcup_iT_i$ is balanced.
\end{theorem}

\noindent Note that the poset associated with $T$ adds completions to all the theories $T_i$ simultaneously.

\begin{proof}
	As a preliminary remark, an application of the disjunctive normal form theorem shows that each formula $\phi$ in the language $\mathcal{L}=\bigcup_i\mathcal{L}_i$ is (provably equivalent to) a disjunction of conjunctions such that each formula in the conjunctions already belongs to one of the languages $\mathcal{L}_i$. 
	
	Suppose that $x, y_0, y_1$ are sets of ordinals such that $\langle T_i\colon i\in I\rangle\in\vd_x$ and $y_0\forkindep_xy_1$. Suppose that $\phi_0\in\vd_{xy_0}$ and $\phi_1\in\vd_{xy_1}$ are such that $T\vdash\phi_0\to\phi_1$. It is sufficient to treat the case $\phi_0=\bigwedge_{i\in a_0}\phi_0^i$ where $a_0\subset I$ is a finite set and each $\phi_0^i\in\mathcal{L}_i$. Suppose that $\phi_1=\bigvee_{m\in n}\bigwedge_{i\in a_1}\phi_1^{im}$, where $a_1\subset I$ is a finite set and $\phi_1^{im}\in\mathcal{L}_i$. Note that $a_0\cap a_1\subset\vd_x$ by Fact~\ref{facticfact}(2). For each $i\in a_0\cap a_1$ and every set $c\subseteq n$, if $T_i\vdash\phi_0^i\to\bigvee_{m\in c}\phi_1^{im}$, then use the balance of the theory $T_i$ to find an interpolant $\psi_{ic}\in\vd_x\cap\mathcal{L}_i$. Let $\psi_i=\bigwedge_c\psi_{ic}$ and $\psi=\bigwedge_{i\in a_0\cap a_1}\psi^i$. We claim that $\psi$ is the required interpolant.
	
	To see this, it is clear that $T\vdash\phi_0\to\psi$. To show that $T\vdash\psi\to\phi_1$, suppose towards a contradiction that this fails. Then, there must be a partition $\{c_i\colon i\in a_0\}$ of $n$ such that for each $i\in a_1$, if $i\in a_0$ then $T_i\cup\{\psi_i\}$ is consistent with $\bigwedge_{m\in c_i}\lnot\phi_1^{im}$ and if $i\notin a_0$ then $T_i$ is consistent with $\bigwedge_{m\in c_i}\lnot\phi_1^{im}$.
	By the choice of $\psi_i$ the former option implies that even $\phi_0^i$ is consistent with $\bigwedge_{m\in c_i}\lnot\phi_1^{im}$ for each $i\in a_0\cap a_1$. This in turn implies that $T\cup\phi_0$ is consistent with $\lnot\phi_1$, a contradiction.
\end{proof}

\begin{example}
There is a Boolean balanced poset adding a system $\langle U_\ga\colon \ga\in\gw_1$ limit$\rangle$ of ultrafilters on $\ga$ containing no bounded subsets of $\ga$. To see this, for each countable limit ordinal $\ga$ let $T_\ga$ be the ultrafilter theory on $\power(\ga)$ modulo the bounded sets. This is balanced by Theorem~\ref{welltheorem}.  Let $T$ be the union of all $S_\ga$ for $\ga$ a countable limit ordinal. Then the associated poset adds a system of ultrafilters, none containing a bounded set.
\end{example}

\section{Imbalanced theories}
\label{imbalancedsection}

\noindent In this section, we present ideals on $\gw$ such that the ultrafilter theory on the quotient Boolean algebra is not balanced. The first is a standard example of a non $F_{\gs}$-ideal, the second is an $F_\sigma$-ideal, and the third is an analytic P-ideal. In all cases, we use the following criterion for the failure of balance:

\begin{theorem}
\label{imbalancetheorem}
Let $I$ be an analytic ideal on $\gw$. Suppose that there are Polish spaces $Y_0, Y_1$ and Borel functions $f_0\colon Y_0\to\power(\gw)$ and $f_1\colon Y_1\to\power(\gw)$ such that the following sets are meager:

\begin{enumerate}
\item the set $\{\langle y_0, y_1\rangle\in Y_0\times Y_1\colon f_0(y_0)\cap f_1(y_1)\notin I\}$;
\item for every set $z\subset\gw$, the set $\{y_0\in Y_0\colon f_0(y_0)\setminus z\in I\}$ or the set $\{y_1\in Y_1\colon f_1(y_1)\cap z\in I\}$.
\end{enumerate}

\noindent Then the ultrafilter theory of $I$ is not balanced.
\end{theorem}

\noindent As a complementary observation, note that the set in (1) is coanalytic and and the sets in (2) are analytic, so their meagerness is a conjunction of analytic and co-analytic statements. Thus, the properties of $f_0, f_1$ are inherited by all well-founded models of ZF+DC which contain Borel codes for $f_0$ and $f_1$.

\begin{proof}
Suppose that $x$ is a set of ordinals such that $I, f_0, f_1$ belong to $\vd_x$, and let $\langle y_0, y_1\rangle\in Y_0\times Y_1$ be a point Cohen-generic over $\vd_x$. Then $y_0\forkindep[x]y_1$, and the sets $f_0(y_0), f_1(y_1)\subset\gw$ belong to the respective classes $\vd_{xy_0}, \vd_{xy_1}$. By the assumption (1) and the genericity, $f_0(y_0)\cap f_1(y_1)\in I$. Thus, writing $a_0=f_0(y_0)$ and $a_1=\gw\setminus f_1(y_1)$, the ultrafilter theory proves $a_0\to a_1$. It will be enough to show that the implication cannot pass through an interpolant in $\vd_x$.

Indeed, if $b\in\vd_x$ is any subset of $\gw$, the set $\{\langle y_0, y_1\rangle\in Y_0\times Y_1\colon a_0\setminus b\in I$ and $a_1\cap b\in I\}$ is in $\vd_x$, it is meager by the assumption (2), and by the genericity of the pair $\langle y_0, y_1\rangle$, it does not contain the pair $\langle y_0, y_1\rangle$. It follows from the definitions that $b$ is not an interpolant between $a_0$ and $a_1$. Since every formula in $\vd_x$ is provably equivalent to an atomic one, this completes the proof.
\end{proof}

\begin{example}
	Let $I$ be the ideal on $\gw\times\gw$ consisting of sets with all vertical sections finite. The ultrafilter theory of $I$ is not balanced.
\end{example}

\begin{proof}
	For the duration of this proof, we replace $\gw\times\gw$ with $\gw\times\bintree$, and define $Y_0=Y_1=(\cantor)^\gw$ and functions $f_0, f_1\colon Y_0, Y_1\to\power(\gw\times\bintree)$ by $f_0(y)=f_1(y)=\{\langle n, t\rangle\colon t\subset y(n)\}$. It is not difficult to verify that the functions $f_0, f_1$ satisfy the demands (1) and (2) of Theorem~\ref{imbalancetheorem}. For (1) note that if $y_0, y_1\in (\cantor)^\gw$ are sequences such that no entry occurs in both $y_0$ and $y_1$ simultaneously, then $f_0(y_0)\cap f_1(y_1)$ has all vertical sections finite. For (2), let $z\subset\gw\times\gwtree$ be any set. Either there are infinitely many $n\in\gw$ such that $z_n\subset\bintree$ is somewhere dense, in which case the set $\{y_1\in Y_1\colon f_1(y_1)\cap z\in I\}$ is meager, or there are only finitely many such $n$, in which case the set $\{y_0\in Y_0\colon f_0(y_0)\setminus z\in I\}$ is meager.
\end{proof}

\begin{example}
Let $c_n$ for $n\in\gw$ be pairwise disjoint finite sets such that $\lim_n |c_n|=\infty$. Let $C=\bigcup_nc_n$ and let $I$ be the ideal on $C$ generated by selectors on the collection $\{c_n\colon n\in\gw\}$. Then the ultrafilter theory of $I$ is not balanced.
\end{example}

\begin{proof}
Passing to a subsequence and shrinking the sets $c_n$ if necessary, we may assume that $|c_n|=(n+1)^2$. For each $n\in\gw$, present $c_n$ as a cartesian product $c_n=a_{0n}\times a_{1n}$ of two sets of cardinality $n+1$. Write $Y_0=\prod_na_{0n}$ and $Y_1=\prod_na_{1n}$. Define a function $f_0\colon Y_0\to\power(C)$ by $f_0(y_0)=\bigcup_n (\{y_0(n)\}\times a_{1n})$. Similarly, define a function $f_1\colon Y_1\to\power(B)$ by $f_1(y_1)=\bigcup_n (a_{0n}\times\{y_1(n)\})$. It will be enough to show that the functions $f_0$ and $f_1$ satisfy the assumptions of Theorem~\ref{imbalancetheorem}.

Clearly, if $y_0\in Y_0$ and $y_1\in Y_1$ are arbitrary points, then $f_0(y_0)\cap f_1(y_1)=\{\langle y_0(n), y_1(n)\rangle\colon n\in\gw\}$ and this set belongs to the ideal $I$ by the definitions, proving the assumption (1). To prove (2), suppose that $z\subseteq C$ is an arbitrary set. Either there are infinitely many numbers $n\in\gw$ such that the relative counting measure of $z\cap a_{0n}\times a_{1n}$ is greater than $1/2$. In such a case, by the Fubini theorem there are infinitely many $n$ for which there is a horizontal section of the set $z\cap a_{0n}\times a_{1n}$ with relative counting measure greater than $1/2$; this means that for comeagerly many $y_1\in Y_1$, $f_1(y_1)\cap z\notin I$. Or, for all but finitely many $n$, the relative counting measure of $z\cap a_{0n}\times a_{1n}$ is at most $1/2$. By a symmetric use of the Fubini theorem it follows that for co-meagerly many $y_0\in Y_0$, $f(y_0)\setminus z\notin I$. (2) of Theorem~\ref{imbalancetheorem} follows.
\end{proof}

\begin{example}
Let $c_n$ for $n\in\gw$ be pairwise disjoint finite sets, each of the equipped with a probability submeasure $\mu_n$ such that masses of singletons tend to $0$ in the set $C=\bigcup_nc_n$. Let $I$ be the ideal of sets $z\subset C$ such that $\lim_n\mu_n(z\cap c_n)=0$. Then the ultrafilter theory of $I$ is not balanced.
\end{example}

\begin{proof}
First note that if $a$ is a finite set equipped with a probability submeasure $\mu$ in which masses of singletons are at most $1/n$, then there is an increasing sequence $\{a_i\colon i\in n+1\}$ of subsets of $a$ such that $|\mu(a_i)-i/n|\leq 1/n$ holds for every $i\in n$: just let $a_0=0$ and recursively let $a_{i+1}$ be an inclusion-minimal superset of $a_i$ whose mass is greater or equal to $i/n$.

This means that passing to a subsequence of the sets $c_n$ if necessary, we may reindex them by elements of $\bintree$ and in each $c_t$ find an inclusion-decreasing sequence $\langle c_{it}\colon i\leq |t|\rangle$ of subsets of $c_t$ such that $1/{i+2}\leq\mu_t(c_{it})\leq 1/{i+1}$. Let $Y_0=Y_1=\cantor$ and let $f_0=f_1\colon Y_0\to \power(C)$ be the function defined by $f_0(y)\cap c_t=c_{it}$ where $i=|\{j\in n\colon y(j)\neq t(j)\}|$. It will be enough to show that the functions $f_0$ and $f_1$ satisfy the assumptions of Theorem~\ref{imbalancetheorem}.

To prove (1) of these assumptions, suppose that $y_0, y_1\in\cantor$ be sequences which differ at infinitely many entries, and argue that $f_0(y_0)\cap f_1(y_1)\in I$. To this end, let $k\in\gw$ be arbitrary. There is a number $n\in\gw$ such that $y_0\restriction n$ and $y_1\restriction n$ differ at $2k$ or more entries. Now, for each finite string $t\in\bintree$ of length at least $n$, one of the sets $\{j\in n\colon y_0(j)\neq t(j)\}$ and $\{j\in n\colon y(j)\neq t(j)\}$ must have cardinality at least $k$, so the intersection $f_0(y_0)\cap f_1(y_1)\cap c_t$ must be a subset of $c_{kt}$, and must have $\mu_t$-mass at most $1/k+1$. This shows that $\lim_t\mu_t(f_0(y_0)\cap f_1(y_1)\cap c_t)=0$ as desired.

To prove (2) of the assumptions, let $z\subset C$ be any set. Suppose that both sets $A_0=\{y_0\in Y_0\colon f_0(y_0)\setminus z\in I\}$ and $A_1=\{y_1\in Y_1\colon f_1(y_1)\cap z\in I\}$  are non-meager; as they are both analytic, there must be strings $u_0, u_1\in\bintree$ such that $A_0$ is co-meager below $u_0$ and $A_1$ is co-meager below $u_1$; we may choose them to be of equal length. Write $n=|\{i\in\dom(u_0)\colon u_0(i)\neq u_1(i)\}|$. By repeated Baire category arguments and the definition of the ideal $I$ it is possible to find a string $v\in\bintree$ and a finite set $d\subset\bintree$ such that the sets $\{y_0\in Y_0\colon\forall w\in\bintree\setminus d\ \mu_w(f_0(y_0)\setminus z)<1/3n\}$ and $\{y_1\in Y_1\colon\forall w\in\bintree\setminus d\ \mu_w(f_1(y_1)\cap z)<1/3n\}$ are comeager below $u_0^\smallfrown v$ and $u_1^\smallfrown v$ respectively, and $u_0^\smallfrown v\notin d$.

Now, write $w=u_0^\smallfrown v$. For every $y_0\in Y_0$ with $u_0^\smallfrown v\subset y_0$, $c_{0w}\subset f_0(y_0)$ holds by the definition of $f_0$. Similarly, for every $y_1\in Y_1$ with $u_1^\smallfrown v\subset y_1$, $c_{nw}\subset f_1(y_1)$. Thus, the sets $f_0(y_0)\cap c_{nw}\setminus z$ and $f_1(y_1)\cap c_{nw}\cap z$ partition the set $c_{nw}$ of $\mu_n$-mass greater than $1/n+2$, and they cannot both have mass smaller than $1/3n$. This contradicts the choice of the string $v$.
\end{proof}

\section{A preservation theorem}
\label{preservationsection}

The main purpose of isolating various classes of balanced forcings is to prove preservation theorems for them, with resulting consistency results in the ZF+DC set theory. In this section, we provide a powerful preservation theorem which shows that Boolean balanced forcings do not add colorings to many hypergraphs which in ZFC are countably chromatic.

\begin{theorem}
	\label{preservationtheorem}
	Let $\Gamma$ be an abelian Polish group, let $d\geq 2$ be a number, and let $H\subset[\Gamma]^d$ be a hypergraph of arity $d$ such that
	
	\begin{enumerate}
		\item $H$ is invariant under shifts;
		\item for every number $i\in\gw$, in every nonempty open set $O\subset\Gamma$ there is a hyperedge $e\in H$ such that for every $i\in j$ $(i+1)e\in H$ holds.
	\end{enumerate}
	
\noindent Then, in every extension by a Boolean balanced forcing, the chromatic number of $H$ is uncountable.
\end{theorem}

\noindent Here, if $k>0$ is a natural number and $e\in H$ is a hyperedge, $ke$ is the set $\{k\gamma\colon \gamma\in e\}$. Note that if the group has torsion elements, this set could have cardinality smaller than $e$. Thus, the second item of the assumption of the theorem is a limited version of invariance of $H$ under re-scaling.

\begin{proof}
The argument starts with a long preamble which deals only with the group and the hypergraph. Fix the group $\Gamma$ and use additive notation for it. Since $\Gamma$ is abelian, any compatible invariant metric on it is complete. Let $d$ be such a metric on $\Gamma$. 
	
	\begin{claim}
		For every number $m$ there is a nonempty finite set $a_m\subset\Gamma$ such that
		
		\begin{enumerate}
			\item $a_m$ consists of points whose distance from $0$ is smaller than $2^{-m}$;
			\item the chromatic number of $H\cap [a_m]^n$ is at least $m$.
		\end{enumerate}
	\end{claim}
	
	\begin{proof}
		First, use the Hales--Jewett theorem \cite{hales:hales} to find a number $j\in\gw$ such that for every partition of $d^j$ into $m$ many classes there is a homogeneous variable word of length $j$. Use the assumption on the hypergraph $H$ to find a hyperedge $e=\{\gamma_l\colon l\in d\}$ consisting of points of distance $<2^{-mj}$ from the neutral element such that for all $i\in j$, $(i+1)e\in H$. Consider the set $a_m=\{\gS_i\delta_i\colon \forall i\in j\ \exists l\in d\delta_i=\gamma_l\}$. We claim that the set $a_m$ works as desired.
		
		To prove this, suppose that $c\colon a_m\to m$ is a function and work to find a $c$-monochromatic hyperedge in $H$. Let $\pi\colon n^j\to m$ be the map defined by $\pi(y)=c(\gS_{i\in j}\gamma_{y(i)})$. Let $u$ be a variable word homogeneous for $\pi$. Let $\gb=\gS_{i\in j}\gamma_{u(i)}$ where the sum is taken over the non-variable entries of the word $u$, and consider $ie$ where $i$ is the number of variables in the word $u$. 
		
		By the choice of the hyperedge $e$, $ie\in H$ holds. By the invariance of $H$ under shifts, $\gb+ie\in H$ holds. And finally, by the homogenity of the variable word $u$, the hyperedge $\gb+ie\subset a_m$ is $c$-monochromatic.
	\end{proof}

\noindent Choose sets $a_m\subset\Gamma$ as in the claim and write $Y=\prod_ma_m$. For every element $y\in Y$, write $\gS y\in \Gamma$ for the sum of all entries on $y$. This sum exists as the entries on the sequence $y$ tend to zero quickly. Let $Q$ be the poset of all functions $q$ whose domain is $\gw$ such that for every $m\in\gw$ $q(m)\subseteq a_m$ is a nonempty set, and the limsup of chromatic numbers of $H\cap [q(m)]^n$ is infinite. The ordering is that of coordinatewise inclusion. The poset adds an element of the product $Y$. Let $R$ be the poset of all infinite subsets of $\gw$ ordered by modulo finite inclusion. The poset adds a nonprincipal ultrafilter on $\gw$. The following claim is central.
		
		\begin{claim}
\label{qrclaim}
		(ZFC) The poset $Q\times R$ forces the filter added by the $R$-coordinate to be an ultrafilter.
		\end{claim}

\begin{proof}
The key point is that the poset $Q$ does not add independent reals by \cite{z:ramsey}. Thus, let $\langle q, r\rangle$ be a condition in the product and $\gs$ is a name for a subset of $\gw$. A standard fusion argument shows that strengthening $q$ and $r$ if necessary we may arrange for $\gs$ to be a $Q$-name. As $Q$ does not add independent reals, there are conditions $q'\leq q$ and $r'\leq r$ such that either $q'\Vdash\gs\cap r'=0$ or $q'\Vdash r'\subset\gs$ holds. So, the condition $\langle q' r'\rangle\leq \langle q, r\rangle$ forces that the generic filter contains a set which is either disjoint from or contained in $\gs$ as desired.
\end{proof}

Finally, we are ready to analyze the Boolean balanced forcing. Let $T$ be a balanced theory in a language $\mathcal{L}$, let $P_T$ be its associated poset, let $p\in P_T$ be a condition, let $\tau$ be a $P_T$-name such that $p\Vdash\tau\colon\Gamma\to\gw$ is a function. Suppose that $x$ is a set of ordinals such that $\vd_x$ contains We need to find a hyperedge $e\in H$ and a condition stronger than $p$ forcing $\tau\restriction\check e$ to be constant.

Let $y\in Y$ be a point and $U$ be a $\vd_x$-ultrafilter on $\gw$ which are $Q\times R$-generic over $\hvd_x$. By Fact~\ref{factgen}(1), $\hvd_{xyU}=\hvd_x[y, U]$ holds, so Claim~\ref{qrclaim} shows that the filter $U$ measures every set ordinally definable from $x, y$, and $U$.
A routine use of the canonical well-ordering of $\vd_{xyU}$ shows that there is a theory in the language $\mathcal{L}\cap\vd_{xyU}$ which is complete in this language, extends $p$, and it is consistent with $T$. Let $F$ be a $\vd$-function such that $F(x, y, U, n)$ is equal to the first such a theory in the canonical definability order derived from $x, y\restriction (\gw\setminus n), U$. Consider $p(y, U)=\{\phi\colon\{n\in\gw\colon \phi\in F(x, y, U, n)\}\in U\}$. Observe that $p(y)\in\vd_{xyU}$, it is a complete theory in the language $\mathcal{L}\cap\vd_{xyU}$, it is consistent with $T$, and it extends $p$. As such, the set of all conditions stronger than $p(y)$ is balanced over $xyU$ by Theorem~\ref{balancetheorem}.

By Fact~\ref{factbal}, there exists a number $k\in\gw$ such that $p(y)\Vdash\tau(\Sigma\check y)=\check k$. By Fact~\ref{factgen}(2), there must be a condition $\langle q, r\rangle\in Q\times R$ in the filter associated with $y, U$ such that any pair $y' ,U'$ which is $Q\times R$-generic over $x$ and meets the condition $\langle q, r\rangle$ satisfies $p(y', U')\Vdash\tau(\Sigma\check y')=\check k$. In the condition $q$, there must be a coordinate $m\in\gamma$ such that the set $q(m)$ contains an $H$-hyperedge $e$. Let $y_i\in Y$ for $i\in d$ be elements equal to $y$ at all coordinates except for $m$, and $e=\{y_i(m)\colon i\in d\}$.  We have the following:

\begin{itemize}
\item each pair $y_i, U$ is $Q\times R$-generic over $\hvd_{x}$ meeting the condition $\langle q, r\rangle$;
\item $\vd_{xyU}=\vd_{xy_iU}$ as the sequences $y$ and $y_i$ differ in only one entry;
\item $p(y_i, U)=p(y, U)$, as the $U$-integration used in the definition of $p(y, U)$ disregards finite differences.
\end{itemize}

\noindent Thus, the points $\gS y_i$ for $i\in d$ form an $H$-hyperedge and the condition $p(y)$ forces all points of this hyperedge to have $\tau$-color equal to $k$. This completes the proof.
\end{proof}

\noindent There are many interesting examples of hypergraphs satisfying the assumptions of Theorem~\ref{preservationtheorem}.

\begin{example}
	\label{vitaliexample}
	Let $\Gamma=\mathbb{R}$ with addition, and $H$ be the graph connecting real numbers with rational distance. Theorem~\ref{preservationtheorem} then shows that in Boolean balanced extensions, $H$ has uncountable chromatic number, which is equivalent to the nonexistence of a Vitali set.
\end{example}

\begin{example}
	Let $\Gamma$ be the countable support product of the cyclic groups $\Gamma_m$ of order $m$ for $m\geq 1$. Let $H$ be the graph connecting two elements of $\Gamma$ if they differ in exactly one entry. This is a diagonal version of the Hamming graph. Theorem~\ref{preservationtheorem} then shows that in Boolean balanced extensions, this graph has uncountable chromatic number.
\end{example}

\begin{example}
	Nonempty hypergraphs on Euclidean spaces invariant under similarities, such as the hypergraph of equilateral triangles on $\mathbb{R}^2$ or of squares on $\mathbb{R}^2$ satisfy the assumptions of Theorem~\ref{preservationtheorem}. Thus, the theorem implies that they all have uncountable chromatic number in Boolean balanced extensions.
\end{example}

\noindent To prove Theorem~\ref{maintheorem}, start in a model of ZFC+there is an inaccessible cardinal $\kappa$. Pass into a $\coll(\gw, <\kappa)$ extension and form the Solovay model as in \cite{solovay:model}. Work under the geometric axiomatization of the Solovay model as in \cite{z:reloadedA}. Consider the set $A$ of all consistent balanced theories which use a locally compact Polish space as their set of propositional variables and are $K_\gs$. Consider the theory $T$ which is the disjoint union of theories in $A$ and force with its associated poset $P_T$. Note the following items:

\begin{itemize}
\item the theory $T$ is balanced by Theorem~\ref{producttheorem};
\item the poset $P_T$ is Boolean balanced, therefore does not add Vitali set by Theorem~\ref{preservationtheorem} and Example~\ref{vitaliexample};
\item every consistent $K_\gs$ theory on a locally compact Polish space is a subset of some theory in the set $A$ by Theorem~\ref{completiontheorem}, so in the extension it has a consistent completion;
\item the poset $P_T$ is $\gs$-closed, therefore preserves DC, and does not add any new $K_\gs$ theories.
\end{itemize}

\noindent To summarize, after the forcing with $P_T$ over the Solovay model, we obtain a model for Theorem~\ref{maintheorem}.

\section{Appendix: axiomatizing the Solovay model}
	\label{appendixsection}
	
The paper \cite{z:reloadedA} contains an axiomatization of the Solovay model relevant for this paper. To start, for every set $z$ of ordinals, the symbol $\vd_z$ stands for the class of all sets definable from $z$ and a finite tuple of ordinals. The symbol $\hvd_z$ stands for the transitive part of $\vd_z$.
	
	\begin{definition}
		The \emph{geometric axiomatization of the Solovay model} consists of the following statements: ZF+DC; every set is definable from a set of ordinals; there is no $\gw_1$-sequence of distinct reals; every set off reals has the property of Baire; and (the geometric axiom) there is a class ternary relation $\forkindep$ on sets of ordinals such that
		
		\begin{enumerate}
			\item (nontriviality) $x \forkindep_x x$ for every set $x$ of ordinals;
			\item (symmetry) $y_0\forkindep_xy_1$ implies $y_1\forkindep_xy_0$;
			\item (heredity) $y_0\forkindep_xy_1$ and $z\in\vd_{xy_0}$ implies $z\forkindep_x y_1$;
			\item (transitivity) $y_0\forkindep_xy_1$ and $z\forkindep_{xy_0}y_1$ implies $y_0z\forkindep_xy_1$;
			\item (extension) for every nonempty $\vd_x$ set $A$ of sets of ordinals, there are $y_0, y_1\in A$ such that $y_0\forkindep_xy_1$;
			\item (product) $y_0\forkindep_x y_1$ and $A\in\vd_x$ and $\langle y_0, y_1\rangle\in A$ implies that there are $\vd_x$ sets $B_0, B_1$ such that $y_0\in B_0$, $y_1\in B_1$, and for every $y'_0\in B_0$ and $y'_1\in B_1$, $y'_0\forkindep_xy'_1$ implies $\langle y'_0, y'_1\rangle\in A$.
		\end{enumerate}
	\end{definition}
	
	\noindent The extension property immediately implies its apparently stronger conclusion: for every set $y_0$ of ordinals there is $y_1\in A$ such that $y_0\forkindep_xy_1$. The rather unwieldy product property says that the ordinal theory of an independent pair of sets of ordinals depends continuously on the theory of the coordinates of the pair.

	\noindent The current paper takes place under this axiomatization. Two consequences for the class of well-orderable sets will be used repeatedly.
	
	\begin{fact}
		\label{facticfact}
		\begin{enumerate}
			\item Let $A$ be a set and $z$ a set of ordinals such that $A\in\vd_z$. Then $A$ is well-orderable iff $A\subset\vd_z$.
			\item if $y_0\forkindep_z y_1$ then $\vd_{zy_0}\cap\vd_{zy_1}=\vd_z$.
			\item if $y_0\forkindep_z y_1$ and $A_0\in \vd_{zy_0}$ and $A_1\in\vd_{zy_1}$ are disjoint subsets of $\vd_z$ then there is $B\in\vd_z$ such that $A_0\subseteq B$ and $B\cap A_1=0$.	\end{enumerate}
	\end{fact}

	\noindent There is also an internal forcing theorem. For a countable poset $P$ and a set $x$ of ordinals such that $P\in\vd_x$, say that a filter $g\subset P$ is generic over $x$ if it meets all dense subsets of $P$ which belong to $\vd_x$. Note that $P\subset\vd_x$ holds by Fact~\ref{facticfact}(2), so there is a $\vd_x$-isomorphism between the ordering $P$ and a relation on ordinals, and it makes sense to speak about the models $\hvd_x[g]$ and $\hvd_{xg}$. The following fact governs the behavior of these classes.

\begin{fact}
\label{factgen}
Let $P$ be a countable poset, $x$ a set of ordinals such that $P\in\vd_x$, and $g\subset P$ be a filter generic over $x$. Then

\begin{enumerate}
\item $\hvd_{xg}=\hvd_x[g]$;
\item whenever $\phi$ is a formula with one free variable and parameters $\bar v$ in $\vd_x$ and $\phi(\bar v, g)$ holds, then there is a condition $p\in g$ such that $\phi(\bar v, h)$ holds for every filter $h\subset P$ generic over $x$ containing the condition $g$.
\end{enumerate}
\end{fact}

\noindent A final remark regarding the geometric axiomatization concerns Polish spaces. If $Y$ is a Polish space and $x$ is a set of ordinals, by $Y\in\vd_x$ we mean that $\vd_x$ contains the underlying set of $X$, a countable dense subset of $Y$ together with its enumeration, and a complete compatible metric $d$ on $Y$. Note that if such is the case, then the metric $d\restriction Y\cap\vd_x$ is $\vd_x$-complete in the sense that for every Cauchy sequence in $\vd_x$ the limit belongs to $\vd_x$ again. This means that $\hvd_x$ contains a separable completely metrizable space which is isometric to a dense subset of $Y$ and the usual Mostowski and Shoenfield absoluteness theorems apply to $\vd_x$ and $Y$ \cite{z:interpretations}.
	
The paper \cite{z:reloadedB} provides a new and streamlined axiomatization of balanced forcing, which is used in this paper. The main definition is the following.
	
	\begin{definition}
		Let $\langle P, \leq\rangle$ be a poset. 
		
		\begin{enumerate}
			\item Let $x$ be a set of ordinals such that $\langle P, \leq \rangle\in \vd_x$. A set $O\subset P$ is \emph{balanced over $x$} if it is a nonempty open subset of $\mathbb{P}$, it belongs to $\vd_x$, and for all sets $y_0, y_1$ of ordinals such that $y_0\forkindep[x] y_1$ and all conditions $p_0\in\vd_{xy_0}$ and $p_1\in\vd_{xy_1}$ in $O$, the conditions $p_0$ and $p_1$ have a common lower bound.
			\item The poset $P$ is balanced if for every set $x$ of ordinals such that $P\in\vd_x$ and every condition $p\in P\cap\vd_x$ there is a set $O\subset P$ which is balanced over $x$ and consists of conditions stronger than $p$.
		\end{enumerate}
	\end{definition}

\noindent The main feature of balanced sets is stated in the following fact.

\begin{fact}
\label{factbal}
Let $P$ be a poset and $x$ be a set of ordinals such that $P\in\vd_x$ and let $O\subset P$ be a set open over $x$. Then any two nonempty open $\vd_x$  subsets of $O$ intersect.
\end{fact}

\noindent As a simple corollary, for any statement $\phi$ of the $P$-forcing language, every condition in $O$ decides $\phi$ and all conditions in $O$ decide it in the same way; also, for any two sets $O_0, O_1$ balanced over $x$, the intersection $O_0\cap O_1$ is either empty or else dense in both $O_0$ and $O_1$
	
	Balanced posets do not add any sets of ordinals, no MAD families of subsets of $\gw$ etc. The full theory is exposed in \cite{z:geometric} and \cite{z:reloadedB}; we refrain from repeating it here.

\section{Declarations}

\subsection{Financial support}

The work in this paper was partly supported by National Science Foundation grant  DMS 2348371.

\subsection{Competing interests}

The authors have no competing interests as defined by Springer, or other interests that might be perceived to influence the results and/or discussion reported in this paper.

\bibliographystyle{plain} 
\bibliography{odkazy,zapletal}

\end{document}